\documentclass[10pt,a4paper]{article}

\usepackage[utf8]{inputenc}
\usepackage[T1]{fontenc}
\pdfoutput=1

\usepackage{amsmath}
\usepackage{amsfonts}
\usepackage{amssymb}
\usepackage{mathtools}

\usepackage{lmodern}
\usepackage{mathrsfs}
\usepackage{bbm}
\usepackage{yfonts}

\usepackage[parfill]{parskip}
\usepackage{array}

\usepackage{graphicx}
\usepackage[font=small,labelfont=bf]{caption}
\usepackage{tikz-cd}

\usepackage{titlesec}
\usepackage{enumerate} 

\usepackage{verbatim}
\usepackage{amsthm}

\usepackage[colorlinks=true, allcolors=blue]{hyperref}
\usepackage[nameinlink]{cleveref}

\newtheorem{theorem}{Theorem}

\newtheorem{corollary}{Corollary}[section]
\newtheorem{lemma}{Lemma}
\newtheorem{prop}{Proposition}

\theoremstyle{remark}
\newtheorem{rem}{Remark}[section]

\theoremstyle{definition}
\newtheorem{definition}{Definition}

\crefname{prop}{proposition}{propositions}
\Crefname{prop}{Proposition}{Propositions}
\crefname{rem}{remark}{remarks}
\Crefname{rem}{Remark}{Remarks}

 \newcommand{\fl}{{\overline{\mathbb{F}}_\ell}}
 \newcommand{\ql}{{\overline{\mathbb{Q}}_\ell}}
  \newcommand{\zl}{{\overline{\mathbb{Z}}_\ell}}
  \newcommand{\rep}{\mathrm{Rep}}

   \newcommand{\repf}{\mathrm{Rep}_{\fl}}
   \newcommand{\ip}{\mathrm{Ind}_P^G}
   \newcommand{\iop}{\mathrm{Ind}_{\overline{P}}^G}
   \newcommand{\df}{\dim_\fl}
   \newcommand{\ho}{\mathrm{Hom}}
   \newcommand{\bl}{\otimes_\zl\fl}
   
   \newcommand{\im}{\mathrm{Im}}
   \newcommand{\op}{\overline{P}}
   \newcommand{\id}{\mathrm{Ind}}
   \newcommand{\xc}{\mathfrak{X}}
   \newcommand{\repr}{\mathrm{Rep}_R}
   \newcommand{\ji}{J_{P,P'}}
   \newcommand{\un}{_{un}}
   \newcommand{\res}{\mathrm{res}}
   \newcommand{\fo}{\mathfrak{o}}
   \newcommand{\ra}{\rightarrow}
   
   \newcommand{\pad}{\mathfrak{d}}
   \newcommand{\ain}[3]{#1\in\{#2,\ldots,#3\}}
   \newcommand{\oo}{\mathbb{O}}
   \newcommand{\tpi}{\Tilde{\pi}}
   \newcommand{\NN}{\mathbb{N}}

   \newcommand{\ed}{\mathrm{End}}
   \newcommand{\trho}{{\Tilde{\rho}}}
   \newcommand{\tsigma}{\Tilde{\sigma}}
   \newcommand{\irr}{\mathrm{Irr}}
   \newcommand{\FF}{\mathrm{F}}
   \newcommand{\mG}{\mathrm{G}   }
   \newcommand{\io}{\mathrm{O}}
   \newcommand{\mM}{\mathrm{M}}
   \newcommand{\mP}{\mathrm{P}}
   \newcommand{\rei}{\rep_{\ql}^{int}}
   \newcommand{\rl}{r_\ell}
   \newcommand{\mvw}{^\mathrm{MVW}}
   \newcommand{\dr}{\mathcal{D}_\rho}
   \newcommand{\rk}{\mathrm{rk}}
   
   \newcommand{\Z}{\mathrm{L}}
   \newcommand{\hra}{\hookrightarrow}
   \newcommand{\cusp}{\mathrm{cusp}}
   \newcommand{\ZZ}{\mathbb{Z}}
   \newcommand{\CC}{\mathbb{C}}
   \newcommand{\tchi}{\Tilde{\chi}}
   \newcommand{\scrO}{\mathscr{O}}
   \newcommand{\cK}{\mathcal{K}}
   \newcommand{\WW}{\mathbb{W}}
   \newcommand{\sra}{\twoheadrightarrow}
   \newcommand{\Sp}{\mathrm{Sp}}
   \newcommand{\Oo}{\mathrm{O}}
   \newcommand{\EE}{\mathrm{E}}
   \newcommand{\tu}{\mathrm{U}}
   \newcommand{\fc}{\mathfrak{c}}
\date{\today}
\begin{document}
\title{Lifting banal representations of classical groups}
\author{Johannes Droschl}
\maketitle
\begin{abstract}
Let $\mG$ be a symplectic, a split orthogonal, or an unramified unitary group over a local non-archimedean field $\FF$. A prime $\ell$ is called banal with respect to $\mG$ if it does not divide the cardinality of the $k$-points of $\mG$, where $k$ is the residue field of $\FF$. In this paper we show that for every banal prime $\ell$, any smooth irreducible $\fl$-representation of $\mG(\FF)$ admits a lift to $\ql$. We also state similar results for more general classical groups of symplectic, orthogonal or unitary type.
As an application we prove Howe-duality in the strongly banal case for symplectic-orthogonal or unitary dual pairs.
\end{abstract}
\section{Introduction}
Let $\FF$ be a local non-archimedean field of characteristic different from $2$ with ring of integers $\io$ and residue field $k$ of cardinality $q$. Let $\ell$ be a prime not dividing $q$ and $\mG$ a reductive group over $\FF$ which we assume in this introduction to be split. The study of representations of $\mG(\FF)$ over various fields of coefficients lies at the heart of the automorphic side of the Local Langlands program. Depending on the chosen field and the group $\mG$, several classification results have been achieved. For example, if $\mG$ is a general linear group, the work of Bernstein and Zelevinsky \cite{BerZel77} settles the case of representations with complex coefficients, and the work of Vignéras, see for example \cite{Vig96}, and Mínguez and Sécherre, \emph{cf}. \cite{MinSec14}, the case of $\fl$-coefficients. In the case of classical groups, the theory of Arthur \cite{arthur2013endoscopic} provides us with a good understanding of the tempered representations with complex coefficients. However, unlike in the case of the general linear group, the behavior of representations of classical groups with coefficients in fields of non-zero characteristic remains largely a mystery.

Taking inspiration from finite groups, the $\fl$-representation theory should be reminiscent of the $\CC$-representation theory, as long as $\ell$ is large in comparison to the size of $\mG$. Such $\ell$ are called \emph{banal}. Over a finite group $W$, such a size condition is usually given by $\ell\nmid\lvert W\rvert$. In this case the $\fl$-representations of $W$ behave exactly as their complex counterparts. In particular, one can realize every irreducible representation in a family over $\zl$, \emph{i.e.} there exists a $\zl[W]$-stable lattice in any irreducible $\ql$-representation, and tensoring this lattice with $\fl$ yields an irreducible representation. Moreover, any irreducible representation over $\fl$ can be constructed in this way.

The main goal of this note is to establish a similar result for symplectic, orthogonal, or unitary groups over non-archimedean local fields. We will assume in this introduction that all groups are split. However, we will treat in the paper also the non-split case. Note that for general linear groups and their inner forms this has been achieved in \cite{minguez2013representations}. However, our methods are quite different and in fact can also be adapted to give a new proof in the general linear case. To be more precise, we call a smooth irreducible $\ql$-representation $\tpi$ of $\mG(\FF)$ a \emph{lift} of an irreducible $\fl$-representation $\pi$ of $\mG(\FF)$ if there exists a $\zl$-lattice $\fo$ inside $\tpi$, in the sense of \cite{Vig96}, such that $\fo\bl\cong\pi$.

We assume that $\mG$ is unramified over $\io$. In this case we can take \cite[Lemma 5.22]{Dat_Helm_Kurinczuk_Moss_2025} as a definition of a banal prime.
\begin{definition}
A prime $\ell$ is called banal with respect to $\mG$ if it does not divide $\lvert \mG(k)\rvert$.
\end{definition}
The main result of the paper is the following, see \Cref{C:main}.
\begin{theorem}\label{T:main}
    Let $G$ be either a symplectic, split orthogonal, or unramified unitary group and $\ell$ a banal prime with respect to $G$.
    Any smooth irreducible $\fl$-representation $\pi$ of $G$ admits a lift to $\ql$.
\end{theorem}
Moreover, we will also explain how a similar result can be achieved for general linear groups and more general classical groups of symplectic, orthogonal, or unitary type.
In \Cref{C:lifting} we also give a sufficient condition involving the cuspidal support for an integral representation over $\ql$ to reduce to an irreducible representation over $\fl$.
The starting point of our explorations will be the following theorem.
\begin{theorem}[{\cite[Proposition 4.15]{Dat_Helm_Kurinczuk_Moss_2024}}]
    \Cref{T:main} is true for cuspidal representations.
\end{theorem}
From here our work builds on the following four pillars, from which in turn the main result follows easily. The main idea is to construct recursively and explicitly all irreducible representations, both over $\fl$ and $\ql$, and check at each step of the iteration that the constructions behave as expected with reduction mod $\ell$. This is achieved using the following  tools.

\textbf{Intertwining operators}

The theory of intertwining operators as developed in \cite{Dat05} over an arbitrary algebraically closed field $R$ over $\ZZ[\frac{1}{q}]$ supplies us with very well behaved morphisms 
between parabolically induced representations. To be more precise, it allows us to define for a parabolic subgroup $\mP\subseteq\mG$ with Levi-factor $\mM$, a parabolic subgroup $\mP'$ with the same Levi-factor $\mM$, and an admissible $R$-representation $\pi$ of $\mM(\FF)$, an intertwining operator between the (normalized) parabolically induced representations
\[J_{\mP(\FF),\mP'(\FF)}(\pi)\colon \id_{\mP(\FF)}^{\mG(\FF)}(\pi)\ra\id_{\mP'(\FF)}^{\mG(\FF)}(\pi).\]
From now on we write for the $\FF$-points of an $\FF$-group $\mG$ just $G$ and in particular $G_n$ for the $
\FF$-points of the general linear group of rank $n$.

Our main result in this regard will be \Cref{L:ratJ}.
Let $\tpi$ be an admissible $\ql$-representation of $M$, $\fo$ be an integral structure of $\tpi$, and write $\pi\coloneqq\fo\bl$. Recall that $\ip(\fo)$ is an integral structure of $\ip(\tpi)$ and write $\op$ for the opposite parabolic subgroup of $P$.
Assume that the following are satisfied.
    \begin{enumerate}
        \item $\df\ho(\ip(\pi),\id_{\op}^G(\pi))=1$.
        \item $J_{P,\op}(\tpi)$ and $J_{P,\op}(\pi)$ are \emph{regular} in the sense of \cite[§7]{Dat05}.
        \item For a certain class of twists of $\tpi$ by $\zl$-valued characters $\chi$ the morphism
  $J_{P,\op}(\tpi\otimes\chi)$ is an isomorphism.  
    \end{enumerate}
\begin{lemma}[{\emph{cf}. \Cref{L:ratJ}}]
     Under the above assumptions \[J_{P,\overline{P}}(\tpi)(\ip(\fo))\subseteq \id_{\op}^G(\fo)\] and \[J_{P,\overline{P}}(\tpi)(\ip(\fo))\bl=\im(J_{P,\op}(\pi)).\]
\end{lemma}
We apply the above lemma to recursively construct representations as the images of intertwining operators and it is here where the banality of $\ell$ plays for the first time a crucial role, as it allows us to verify the last condition.
In combination with the next ingredient, the above lemma allows us to reduce the question of lifting to a certain class of \emph{tempered} representations.

\textbf{Derivatives}

We recall the theory of derivatives for classical groups, as their counterpart for general linear groups is slightly less involved. The theory is originally due to \cite{Jader} and \cite{Min08}. We follow the exposition in \cite{atobe2024local}.
Let $\rho$ be an irreducible cuspidal representation of a general linear group $G_m$ over either $\fl$ or $\ql$. An irreducible smooth representation $\pi$ of $G$ is called $\rho$-reduced if either $m>\rk(G)$ or $r_{P}(\pi)$ does not contain a $(G_m,\rho)$-isotypic subquotient. Here $r_P$ denotes the Jacquet-functor with respect to a suitable parabolic subgroup.
We also write for a representation $\rho$ of $G_m$ and $k\in\NN$,  $
\rho^k$ for the parabolically induced representation of $\rho\otimes\ldots\otimes\rho$ to $G_{km}$. If $\ell$ is a banal prime for $G_m$ and $\rho$ is cuspidal, then $\rho^k$ is irreducible.
We also denote by $\rho^\lor$ the dual representation.
\begin{lemma}
    Let $\rho$ and $\pi$ be as above and assume that $\rho$ is non-self-dual. Then there exist $d_\rho\in\NN$ and $\dr(\pi)$, a $\rho$-reduced irreducible smooth representation, such that the following holds.
    \begin{enumerate}
        \item $J_{P,\op}((\rho^\lor)^{d_\rho}\otimes\dr(\pi))$ is regular.
        \item The image of the intertwining operator is isomorphic to $\pi$.
    \end{enumerate}
    \end{lemma}
In the case of the general linear group, the corresponding theorem is already enough to prove \Cref{T:main}. In the case of the classical groups, we need two more ingredients.
    
\textbf{Representations of Arthur type}

We use the following results, which stem from the Arthur classification of tempered representations of $G$.
\begin{lemma}[{\cite[§5.3]{AtobeMinguez2020}}]
    Let $\pi$ be an irreducible smooth tempered $\CC$-representation of $G$ which is $\rho$-reduced for all non-self-dual cuspidal representations $\rho$.
    Then there exists a parabolic subgroup $P$ of $G$ such that $\pi$ is a subrepresentation of $\ip(\sigma)$, where $\sigma$ is a self-dual, cuspidal representation of the Levi-factor $M$ of $P$.
\end{lemma}
The above Lemma thus quickly reduces the claim to understanding $\ip(\sigma)$, where $\sigma$ is a self-dual, cuspidal representation of $M$ over $\fl$ or $\ql$. 

\textbf{Progenerators}

To understand representations of the above form, we use 
the work of \cite{Heiermann2011} and \cite{Roche_2002}. Let $M^0$ be the intersection of the kernels of all unramified characters of $M$. Let $\sigma'$ be a summand of the restriction of $\sigma$ to $M^0$.
It is possible to understand $\ip(\sigma)$ thanks to an observation originally due to Bernstein, see also \cite{Roche_2002}. Namely, \[\ed(\ip(\id_{M^0}^M(\sigma')))\] is a progenerator of a certain block of the category of representations of $G$. Since $\ell$ is banal, and hence any cuspidal representation is projective in its respective block, this property holds also over $\fl$.
We then use the description of \cite{Heiermann2011} of \[\ed(\ip(\id_{M^0}^M(\sigma')))\] in terms of intertwining operators to get the desired description and finish the proof of the main theorem.

As an application of the above tools we have the following proposition.
\begin{prop}
    Assume $\tpi,\tpi'$ are two integral irreducible smooth $\ql$-\hspace{0pt}representations of $G$ whose reductions mod $\ell$ are irreducible and isomorphic. If $\ell$ is banal and one of them is an essentially discrete series representation, so is the other.
\end{prop}
In the last section, we give a proof of the modular Howe-duality conjecture in the strongly banal case for symplectic-orthogonal and unitary dual pairs. The modular local theta correspondence has been studied and developed in great detail in \cite{Trias2025}, where the author managed to prove the respective statements of Howe-duality under a certain hypothesis which he calls \emph{Hypothesis (H)} and which in turn boils down to the question whether a certain principal series representation is irreducible. Note that in \cite{Trias2025} Hypothesis (H) was proven for all but finitely many primes $\ell$. However, no explicit descriptions of those that satisfy it were known.
Thanks to the work of \cite{Trias2025} it will thus suffice to prove that the reduction mod $\ell$ of a certain integral irreducible representation is irreducible.
Using the above results, we can achieve this without much effort and hence complete the last missing step in the proof of the modular theta correspondence in the strongly banal case as laid out in \cite{Trias2025}. For more details see \Cref{S:Howe}.

Let us finally quickly note that in the ramified case, we will define a class of \emph{banal} representations over $\fl$ and show in \Cref{C:main} that any banal representation admits a lift to $\ql$. To be more precise, a representation is called \emph{banal} if its cuspidal support can be lifted in a \emph{banal} way, a notion that is inspired by the methods of \cite{minguez2013representations}. For more details see \Cref{S:inrep}. In particular, if $\ell$ is banal with respect to $G$, all irreducible representations of $G$ are banal by \Cref{L:cuspsupext}.
\subsection*{Acknowledgements}
I want to express my gratitude towards Alberto Mínguez for pointing out the problem, suggesting that it might be amenable to some techniques I recently learned, as well as for his never-ending support. I  want to thank Erez Lapid for the useful discussions and pointing me towards the work of \cite{Heiermann2011}. Finally, I want to thank Justin Trias and Vincent Sécherre for providing valuable feedback on earlier versions of this paper.
The author was supported by a research grant (VIL53023) from VILLUM FONDEN.
\numberwithin{theorem}{section}
\numberwithin{prop}{section}
\numberwithin{lemma}{section}
\section{Preliminaries}
 In this paper we denote the symplectic group of rank $n$ over $\FF$ by $\Sp_n$ and the split orthogonal group in dimension $n$ over $\FF$ by $\Oo_n$.
 More generally, let $\epsilon\in\{\pm 1\}$, $\EE$ be either an extension of degree two of $\FF$ or $\FF$ itself, and let $W$ be an $\epsilon$-hermitian $\EE$-vector space. By a \emph{classical group} we will mean in this paper a symmetry group $\mG(W)$ of such a vector space, \emph{i.e.} a classical group of symplectic, orthogonal or unitary type. If $\EE$ is an unramified quadratic extension of $\FF$, we write $\tu_n$ for the unique unramified unitary group over $\FF$ of rank $n$ attached to $\EE$.
 As in the introduction, we will write $G_n=\mathrm{GL}_n(\EE).$
 
 Let $N_G$ be the least-common-multiple of the pro-orders of the open compact subgroups of $G$ in the sense of \cite{Vig96}. We call a prime $\ell$ banal if it does not divide $N_G$.
\begin{lemma}[{\cite[Lemma 5.22]{Dat_Helm_Kurinczuk_Moss_2025}}]\label{L:banalcrit}
    Let $G=G_n,\,\Sp_n,\,\Oo_n$ or $\tu_n$.
    A prime $\ell$ is banal for $G$ if and only if $\ell$ does not divide  $\lvert \mG(k)\rvert$.
\end{lemma} 
We recall the well-known formulas for the cardinality of $\mG(k)$.
\[\lvert\mathrm{GL}_n(k)\rvert=\prod_{i=0}^{n-1}(q^n-q^i),\,\lvert\Sp_n(k)\rvert=q^{n^2}\prod_{i=1}^n(q^{2i}-1),\]\[ \lvert\Oo_{2n+1}(k)\rvert=2q^{n^2}\prod_{i=1}^n(q^{2i}-1),\,\lvert\Oo_{2n}(k)\rvert=2q^{n^2-n}(q^n-1)\prod_{i=1}^{n-1}(q^{2i}-1),\]
\[\lvert\tu_n(k)\rvert=q^{\frac{n(n-1)}{2}}\prod_{i=1}^n(q^i-(-1)^i).\]
If $\mM$ is an $\FF$-Levi-subgroup of an $\FF$-parabolic subgroup $\mP$ of $\mG$ and $\ell$ is a banal prime for $G$, it is also banal for $M$.
We will call $P$ a parabolic subgroup of $G$ if there exists an $\FF$-parabolic subgroup $\mP$ of $\mG$ such that $P=\mP(\FF)$ and similarly for the Levi-factor.

We fix for the rest of the paper a prime $\ell$ which does not divide $q$ and denote by $R$ one of the algebraically closed fields $\fl$ or $\ql$. Most of the time we will assume that $\ell$ is banal with respect to $G$. However, some results can be stated in greater generality. We write $o(q)$ for the order of $q\in\fl$ and write $\rep_R(G)$ for the category of admissible and smooth $R$-representations of $G$. We also fix an isomorphism $\ql\ra \CC$ and note that $\pi\in\rep_\ql(G)$ is irreducible if and only if $\pi\otimes_\ql\CC\in\rep_\CC(G)$ is irreducible. We denote by $\irr_R(G)$ the set of isomorphism classes of irreducible representations in $\rep_R(G)$.
We denote the Grothendieck group of the full subcategory of finite length representations by $K_R(G)$ and use as usual $[-]$ to denote the image of a suitable representation in the latter.

Throughout the following paragraphs we recall some structural properties of the category $\rep_R(G)$ from \cite{Vig96}.

If $P$ is a parabolic subgroup of $G$ and $M$ its Levi-factor there exist the exact functors of normalized parabolic induction and the normalized Jacquet-functor
\[\ip\colon \rep_R(M)\ra\rep_R(G),\, r_P\colon\rep_R(G)\ra\rep_R(M).\] We choose the square-roots of $q$ in $\fl$ and $\ql$ in the definition of the modular character $\delta_P^{\frac{1}{2}}$ such that they are compatible with reduction $\mod\ell$. Recall that $r_P$ is the left adjoint of $\ip$ by Frobenius reciprocity.

An irreducible representation $\rho\in\irr_R(G)$ is called cuspidal if for all non-trivial parabolic subgroups $P$ of $G$ we have $r_P(\rho)=0$. Since $\ell$ is banal, this is equivalent to $\rho$ being supercuspidal, \emph{i.e.} there exists no non-trivial parabolic subgroup $P$ with Levi-factor $M$ and $\pi\in \rep_R(M)$ such that $\rho$ is a subquotient of $\ip(\pi)$, see for example \cite[Theorem, p.~373]{Vig96}. Moreover, every cuspidal representation is projective in the subcategory of the respective block with fixed central character.
We denote the subset of cuspidal representations of $\irr_R(G)$ by $\irr_{R,c}(G)$.
For $\rho\in\irr_{R,c}(G_n)$, we write $o(\rho)$ for the cardinality $\lvert\{[\rho\otimes\lvert\det\rvert^k]:k\in\ZZ\}\rvert$ and recall that $o(\rho)$ equals the order of $q^{f(\rho)}$ in $\fl$, where $f(\rho)$ is a certain divisor of $n$, \emph{cf}. \cite[III]{Vig96}. In particular $f(\rho)\le n$.
Finally, for $\rho\in\irr_{R,c}(G_n)$, we let $\{[\rho\otimes\lvert\det\rvert^k]:k\in\ZZ\}$ be the corresponding cuspidal line and $\{[\rho\otimes\lvert\det\rvert^\frac{k}{2}]:k\in\ZZ\}$ the corresponding cuspidal half-line.

We let $d_G$ be $n$ if $G=G_n$ and the dimension of the maximal isotropic subspace of $W$ if $G$ is classical. If $G$ is a general linear group, we call a parabolic subgroup $P$ of $G$ standard if it contains the Borel subgroup of upper-triangular matrices. If $G$ is classical, we fix a maximal flag of isotropic subspaces and call a parabolic subgroup standard if it contains the stabilizer of this maximal flag.
If $G=G_n$, we then identify the set of standard parabolic subgroups with the (ordered) partitions $
\alpha=(\alpha_1,\ldots,\alpha_k)$ of $n$. If $\pi=\pi_1\otimes\ldots\otimes\pi_k$ is a representation of the Levi-subgroup $M_\alpha$ of $P=P_\alpha$, we write 
\[\pi_1\times\ldots\times\pi_k\coloneqq\ip(\pi).\]

For a classical group, the standard parabolic subgroups are in bijection with (ordered) partitions $\alpha=(\alpha_1,\ldots,\alpha_k)$ of some $r\le d_G$. In this case we write for a representation $\pi=\rho_1\otimes\ldots\otimes\rho_k\otimes \sigma$ of the Levi-factor $M_\alpha=G_{\alpha_1}\times\ldots\times G_{\alpha_k}\times G'$, where $G'$ is another classical group of the same type,
\[\rho_1\times\ldots\times\rho_k\rtimes\sigma\coloneqq\ip(\pi).\]
For $G$ either general linear or classical we write $r_{P_\alpha}=r_\alpha$. We set for $m\le d_G$
\[d_{W,m}\coloneqq \begin{cases}
    \frac{\dim_{\EE} W-2m-\epsilon}{2m}&\EE=\FF,\,\dim_\EE W\text{ even},\\
    \frac{\dim_{\EE} W-2m}{2m}&\EE=\FF,\,\dim_\EE W\text{ odd},\\
    \frac{\dim_{\EE} W-2m+1}{2m}&[\EE:\FF]=2.\end{cases}\]
These numbers will be used to bound $s\in\frac{1}{2}\ZZ$ such that $\rho\otimes \lvert\det\rvert^s\rtimes \sigma$ is irreducible for $\rho$ and $\sigma$ suitable cuspidal representations. 
    
We also fix a generator $\fc\in\mathrm{Gal}(\EE/\FF)$ and denote for $\pi\in\rep_R(G_n)$ by ${}^\fc\pi$ the $\fc$-twist of $\pi$. We call $\pi$ \emph{$\fc$-self-dual} if $\pi^\lor\cong{}^\fc\pi$.
    
We also denote by $W=W(G)$ the Weyl group of $G$. If $\alpha$ is a partition as above, consider the stabilizer $W^\alpha$ of $M_\alpha$ in $W$. We define the Weyl-group $W_\alpha\subseteq W$ as the set of minimal (with respect to their length) representatives of the $W(M_\alpha)$-cosets of $W^\alpha$.
\subsection{Geometric Lemma}\label{S:GL}
We now recall the Geometric Lemma of \cite{BerZel77}.
Let $G$ be as above and $P$ and $Q$ two parabolic subgroups of $G$ with Levi-factors $M$ and $N$ and unipotent-factors $U$ and $V$. Moreover, let $\pi\in\rep_{R}(M)$.
Choose an order $\oo_1,\ldots,\oo_k$ of the $Q$-orbits on $P\backslash G$ such that $\oo_j\subseteq \overline{\oo_i}$ implies that $i\le j$.
Let $F(\oo_i)(\pi)$ be the $Q$-invariant subset of $\ip(\pi)$ consisting of those functions whose support is contained in $\bigcup_{j=1}^i\oo_j$.
Moreover, we define for each $\ain{i}{1}{k}$ a representation $\sigma_i(\pi)$ as follows.
Define for $\oo_i$ the groups
 \[M'\coloneqq M\cap w^{-1}Nw,\, N'\coloneqq wM'w^{-1},\,V'\coloneqq M\cap w^{-1}Vw,\, U'\coloneqq N\cap w Uw^{-1},\]
where $w$ is a representative of $\oo_i$. Finally, let $P'\coloneqq N'U'$ and $Q'\coloneqq M'V'$. 
We then set $\sigma_i(\pi)\coloneqq \id_{P'}^{N}\circ \mathrm{Ad}(w)\circ r_{Q'}(\pi)$.
\begin{lemma}[{\cite[Theorem 2.12]{BerZel77}}]
    The filtration \[0=r_Q(F(\oo_0)(\pi))\subseteq r_Q(F(\oo_1)(\pi))\subseteq r_Q(F(\oo_2)(\pi))\subseteq \ldots\subseteq r_Q(F(\oo_k)(\pi))\] is a functorial filtration of $r_Q(\ip(\pi))$ whose subquotients are of the form 
    \[r_Q(F(\oo_{i-1})(\pi))\backslash r_Q(F(\oo_i)(\pi))\cong \sigma_i(\pi)\]
    for $\ain{i}{1}{k}$.
\end{lemma}
We will demonstrate how the above formula can be used to express the semi-simplification of the Jacquet-functor in two cases.

Firstly, in the case of general linear groups, the situation is as follows.
Let $\pi_1\in\rep_R(G_{n_1})$, $\pi_2\in\rep_R(G_{n_2})$ and $\alpha=(\alpha_1,\alpha_2)$ a partition of $n_1+n_2$. For $\ain{k}{0}{\min(\alpha_1,n_1)}$, $n_2\ge \alpha_1-k$, we consider the semi-simplification \[[r_{k,n_1-k}(\pi_1)]=\sum [\sigma_{1,k}\otimes\sigma_{2,k}],\,[r_{\alpha_1-k,n_2-\alpha_1+k}(\pi_2)]=\sum [\sigma_{1,k}'\otimes\sigma_{2,k}'],\] where $\sigma_{i,k},\sigma_{i,k}',i=1,2$ are irreducible representations. Then
\[[r_{\alpha_1,\alpha_2}(\pi_1\times\pi_2)]=\sum_k\sum[\sigma_{1,k}\times\sigma_{1,k}']\otimes [\sigma_{2,k}\times\sigma_{2,k}'].\]
In the case of classical groups, the following formula is known as Tadi\'c's formula.
Let $\pi\in\rep_R(G_n)$, $\tau\in \rep_R(G)$ with $G$ a classical group, and $\alpha\in\ZZ_{\ge 0}$, $d_G\ge \alpha+n$. For $\ain{i}{0}{n},\ain{j}{0}{\min(\alpha,d_G)}$ with $i+j\le \alpha$ and $n\ge \alpha-j$ we write
\[[r_{i,n-\alpha+j,\alpha-i-j}(\pi)]=\sum[\sigma_{i,j,1}\otimes\sigma_{i,j,2}\otimes\sigma_{i,j,3}],\, [r_{j}(\tau)]=\sum [\sigma_{j}'\otimes\tau_{j}], \]where $\sigma_{i,j,k},\sigma_{j}',\tau_{j},k=1,2,3$ are irreducible representations.
Then
\[[r_{\alpha}(\pi\rtimes\tau)]=\sum_{i,j}\sum [\sigma_{i,j,1}\times \sigma_{j}'\times {}^\fc\sigma_{i,j,3}\otimes \sigma_{i,j,2}\rtimes\tau_{j}].\]
\subsection{Integral representations}\label{S:inrep}
We call a representation $\pi\in\rep_{\ql}(G)$ integral if it admits a $\zl$-integral structure in the sense of \cite{Vig96}. Denote by $\rei(G)$ the category of pairs $(\pi,\fo)$, where $\pi$ is integral and $\fo$ is an integral structure of $\pi$. Morphisms are morphisms in $\rep_\ql(G)$ preserving the integral structure. We denote the reduction $\mod \ell$ functor which sends $(\pi,\fo)\mapsto \fo\bl$ by $r_\ell\colon \rei(G)\ra\repf(G)$. Moreover, if $(\pi,\fo)$ is of finite length, we have  that $[r_\ell(\pi)]$ only depends on $\pi$ and not on the choice of integral structure by the Brauer-Nesbitt principle of \cite{Vig96}. We thus write for any finite length integral representation $\pi$ suggestively $r_\ell([\pi])\coloneqq [r_\ell(\pi)]$.
We also recall from \cite[I 9.3]{Vig96} that if $(\pi,\fo)\in \rei(M)$, then $(\ip(\pi),\ip(\fo))\in\rei(G)$ and $r_\ell$ commutes with parabolic induction.  
On the level of representations, $r_P$ does not commute with $\rl$. However, after passing to Grothendieck groups, it does, \emph{i.e.} for $\pi$ a finite length integral representation of $G$, $r_P(\pi)$ is also integral and
\[r_P(\rl[\pi])=[\rl(r_P(\pi))],\]
see  \cite[Proposition 6.7]{Dat05}.
 If $\pi\in\irr_\fl(G)$, we call $\tpi\in\irr_\ql(G)$ a \emph{lift} of $\pi$ if there exists an integral structure $\fo$ on $\tpi$ such that $\fo\bl\cong\pi$. Note that by the above properties, if $\tpi$ is a lift of $\pi$, then any integral structure of $\tpi$ reduces to $\pi$.

We now recall the following theorem, which will provide the foundation for the results of this note.
\begin{theorem}[{\cite[Proposition 4.15]{Dat_Helm_Kurinczuk_Moss_2024}}]\label{T:cuslift}
    Let $G$ be a classical or general linear group and $\ell$ a banal prime.
    Then $\trho\in\irr_{\ql,c}(G)$ admits an integral structure if and only if its central character is $\zl$-valued. Moreover, in this case $\rl([\trho])$ is irreducible.
    Finally, any $\rho\in\irr_{\fl,c}(G)$ admits a lift $\trho$ to $\ql$.
\end{theorem}
Recall that if $\ell$ is banal, any $\pi\in\irr_R(G)$ has a well-defined cuspidal (which equals the supercuspidal) support, see \cite[II 2.20]{Vig96}. Namely, if $G$ is a general linear group $\cusp(\pi)\coloneqq [\rho_1]+\ldots+[\rho_k]$, where $\rho_1\otimes\ldots\otimes\rho_k$ is a cuspidal representation of a Levi-subgroup of a parabolic subgroup $P$ of $G$
such that $\pi\hra \ip(\rho_1\otimes\ldots\otimes\rho_k)$.
In the classical case it is defined as \[\cusp(\pi)\coloneqq [\rho_1]+\ldots+[\rho_k]+[{}^\fc\rho_1^\lor]+\ldots+[{}^\fc\rho_k^\lor]-\sum_{\rho_i\,\fc\text{-self dual}}[\rho_i]+[\sigma],\] where $\rho_1\otimes\ldots\otimes\rho_k\otimes\sigma$ is a cuspidal representation of a Levi-subgroup of a parabolic subgroup $P$ of $G$
such that
\[\pi\hra \ip(\rho_1\otimes\ldots\otimes\rho_k\otimes\sigma).\]
The following is an easy consequence of the fact that the reduction of an integral representation is integral.
\begin{lemma}
    Let $\tpi\in\irr_\ql(G)$ be integral. Then the cuspidal support consists of integral representations.
\end{lemma}
We say that an integral representation $\tpi\in\irr_\ql(G)$ has \emph{banal} cuspidal support if the following holds for all representations $\trho$ of $G_m$ in the cuspidal support.
\begin{enumerate}
    \item (General linear group) If $[\trho']\in \cusp(\tpi)$ and $\rl([\trho])$ is in the same cuspidal line as $\rl([\trho'])$, then the same is true for $\trho$ and $\trho'$. Moreover, $\trho'\otimes\lvert\det\rvert^k$ is not isomorphic to $\trho\otimes\lvert\det\rvert^\epsilon$, where $k$ is a non-zero integer multiple of $o(\rl([\trho]))$ and $\epsilon\in\{-1,0,1\}$.
    \item (Classical group) If $\rl([\trho])$ is $\fc$-self-dual, then so is $\trho$.
    \item (Classical group) If $[\trho']\in \cusp(\tpi)$ is a second representation such that $\rl([\trho])=\rl([{}^\fc\trho'])$, then $\trho\cong{}^\fc\trho'$.
    \item (Classical group) If $[\trho']\in \cusp(\tpi)$ and $\rl([\trho])$ is in the same cuspidal half-line as $\rl([\trho'])$, then the same is true for $\trho$ and $\trho'$. Moreover, $\trho'\otimes\lvert\det\rvert^k$ is not isomorphic to $\trho\otimes\lvert\det\rvert^\epsilon$, where $k$ is a non-zero integer multiple of $o(\rl([\trho]))$ and $\epsilon\in\{-1,0,1\}$.
    \item (Classical group) If $\rl([\trho])$ is of the form $\rho'\otimes\lvert\det\rvert^\frac{k}{2}$, $k\in \ZZ$ and $\rho'$ $\fc$-self-dual, then $\trho\cong\trho'\otimes\lvert\det\rvert^\frac{k}{2}$ with $k\in\ZZ$, $\trho'$ $\fc$-self-dual, and for all non-zero multiples $k'$ of $o(\rl([\trho]))$
    \[\lvert k+2k'\rvert>2d_{W,m}.\]
\end{enumerate}
 Finally, a banal lift of a cuspidal support of some $\pi\in\irr_\fl(G)$ is a formal sum of lifts of the cuspidal representations appearing in it that are banal in the above sense. 
 For the split general linear, symplectic, and orthogonal groups, the above construction ties into the definition of a banal prime as follows.
\begin{lemma}\label{L:cuspsupext}
    Let $\pi\in\irr_\fl(G)$ and $\mG$ either $\mathrm{GL}_n,\, \Sp_n,\,\Oo_n$ or $\tu_n$ and $\ell$ a banal prime. Then there exists a banal lift of $\cusp(\pi)$. Moreover, for any banal lift of $\cusp(\pi)$, there exists an integral $\tpi\in\irr_\ql(G)$ with that banal cuspidal support such that $\rl([\tpi])$ contains $[\pi]$.
\end{lemma}
\begin{proof}
The first claim follows from \Cref{L:banalcrit} and the well known formulas for $\lvert\mG(k)\rvert$.
Indeed, if $G$ is a general linear group, we proceed as follows. Let first $\rho_1,\ldots,\rho_j$ be the representations in $\cusp(\pi)$ in some cuspidal line. By the banality of $\ell$ we can assume without loss of generality that $\rho_1\otimes \lvert\det\rvert^{-1}$ does not appear in $\cusp(\pi)$. We then choose a lift $\trho_1$ of $\rho_1$ and fix a lift $\trho_i$ of $\rho_i$ by writing $\rho_i=\rho_1\otimes\lvert\det\rvert^{k_i}$, $k_i\in\{0,\ldots,o(\rho_1)-2\}$ and setting $\trho_i\coloneqq \trho_1\otimes\lvert\det\rvert^{k_i}$. For classical groups the argument is similar. Indeed, one can construct easily as above for each set $\rho_1,\ldots,\rho_j$ of representations in $\cusp(\pi)$ in some cuspidal line lifts that satisfy conditions (2), (3), and (4). For (5), we will only show the claim in the symplectic case. The orthogonal and unitary cases follow analogously. We assume without loss of generality that $\rho_1$ is such that $\rho_1\otimes\lvert\det\rvert^{-1}$ does not appear in the cuspidal support and denote by $\trho_i$ the lift of $\rho_i$ and let $k_i$ be as in condition (5). We then twist, if necessary, all $\trho_i$ by the same character $\lvert\det\rvert^{k'}$, where $k'$ is some integer multiple of $o(\rho_1)$ such that $-\frac{2n-m+1}{2m}\le \frac{k_i}{2}$ and $k_1$ is minimal with that property. Since $o(\rho_1)m> 2n$ by the banality of $\ell$, the lift satisfies also property (5).

    The second claim is immediate.    
\end{proof}
We now make the following definition. 
Let $\pi\in\irr_\fl(G)$ and assume that $\ell$ is banal for $G$. We say $\pi$ is banal if its cuspidal support admits a banal lift. Note that by the above Lemma, in the considered unramified cases this applies to all irreducible representations if $\ell$ is banal.
\subsection{Langlands datum}
Assume that $R=\ql$ and recall the fixed isomorphism $\ql\ra \CC$. 
A character $\chi=\chi_1\otimes\ldots\otimes \chi_k$ of 
the Levi subgroup $M_\alpha,\,\alpha=(\alpha_1,\ldots,\alpha_k),$ of $G_n$ is called positive if
the real parts of the characters satisfy $\mathrm{Re}(\chi_1)\ge \ldots\ge \mathrm{Re}(\chi_k)$.
Similarly, a character $\chi=\chi_1\otimes\ldots\otimes \chi_k$ of 
the Levi subgroup $M_\alpha,\,\alpha=(\alpha_1,\ldots,\alpha_k)$, of a classical group $G$ is called positive if
the real parts of the characters satisfy $\mathrm{Re}(\chi_1)\ge \ldots\ge \mathrm{Re}(\chi_k)\ge 0.$ In both cases, we call the character strictly positive if all $\ge$ are replaced by $>$.

By Casselman's criterion, we can define a representation $\pi\in\irr_\ql(G)$ to be essentially tempered if for any $P_\alpha\subseteq G$ and $\chi$ an exponent of $r_\alpha(\pi)$, $\chi$ is positive. If $\pi$ has moreover a unitary central character, we call $\pi$ tempered.
If all exponents are strictly positive, we call $\pi$ an essentially discrete series representation, and if its central character is unitary, we call it a discrete series representation.

We now recall the classical classification theorem of Langlands, see for example \cite{Borel_Wallach_2000}.
\begin{theorem}[Langlands classification for $p$-adic groups]
    Let $G$ be either a general linear group or a classical group and $\pi\in\irr_\ql(G)$. Then there exists a unique triple $(P_\alpha,\chi,\sigma)$, where $\chi$ is a positive character and $\sigma$ a tempered irreducible representation of $M_\alpha$, such that $\pi$ is a quotient of $S(\pi)\coloneqq\id_{P_\alpha}^G(\sigma\otimes\chi)$ and a subrepresentation of $S'(\pi)\coloneqq \id_{\overline{P_\alpha}}^G(\sigma\otimes\chi)$. Moreover,
    \[\dim_\ql\ho(S(\pi),S'(\pi))=1,\]
    $\pi$ appears with multiplicity $1$ in $S(\pi)$ and is the unique irreducible quotient of $S(\pi)$.
\end{theorem}
\subsection{MVW-involutions}
Let us now recall the MVW-involution of $\rep_R(G)$ of \cite{Waldspurger_1987}, where $G$ is a classical group. We state the properties relevant for our considerations here.
\begin{prop}\label{P:MVW}
    There exists a covariant involution 
    \[(-)\mvw\colon \rep_R(G)\ra \rep_R(G)\] with the following properties.
    \begin{enumerate}
        \item If $\pi\in\irr_R(G)$ then $\pi\mvw\cong\pi^\lor$.
        \item If $\rho_1\otimes\ldots\otimes\rho_k\otimes \sigma$ is a representation of a standard Levi-factor $M$ of $G$, then 
        \[(\rho_1\times\ldots\times \rho_k\rtimes\sigma)\mvw\cong {}^\fc\rho_1\times\ldots\times{}^\fc\rho_k\rtimes \sigma\mvw.\]
    \end{enumerate}
\end{prop}
\section{Intertwining operators}
We fix for this section a standard parabolic subgroup $P$ of $G$ with Levi-factor $M$ and $P'$ a parabolic subgroup with the same Levi-factor. Let $\xc_R[M]$ be the ring of unramified characters of $M$ for $R\in\{\fl,\ql\}$ and $\cK$ its quotient-field. We let $M^0$ be the intersection of all kernels of unramified characters of $M$, hence $\xc_R[M]=R[M^0\backslash M]$.
Let $\pi\in\repr(M)$, let $\psi\un\colon M\ra \cK$ be the tautological character and denote for $\pi\in \repr(M)$ by $\pi\un\coloneqq \pi\otimes_R\cK\otimes\psi\un$ the base change.
Following \cite[§7]{Dat05} and \cite[§IV]{waldspurger2003plancherel} we recall the rational intertwining operators
\[J_{P,P'}(\pi\un)=J_{P,P'}^G(\pi\un)\colon \ip(\pi\un)\ra\id_{P'}^G(\pi\un).\]
Note that part of the construction of \cite[§7.3]{Dat05} is a certain additive character $M^0\backslash M\ra \ZZ$, which can be interpreted as giving a ring-map \[\nu\colon \xc_R[M]\ra R[T,T^{-1}].\] We fix this character once and for all and assume that it is integral and contained in the negative fundamental chamber corresponding to the relative root system of $P$.
These morphisms are rational in the following sense.
By the Iwasawa-decomposition, we have $G=PK,\,G=P'K,\,K\coloneqq \mathrm{G}(\io)$. For $\pi\in\repr(M)$, we have a natural restriction map $\res_K(\pi)\colon \ip(\pi)\ra\id_{P\cap K}^K(\pi)$, which is an isomorphism and has the property that for any unramified character $\chi$ of $M$
$\id_{P\cap K}^K(\pi\otimes\chi)=\id_{P\cap K}^K(\pi)$ and $\id_{P\cap K}^K(\pi\un)=\id_{P\cap K}^K(\pi)\otimes_R \cK$.
By abuse of notation we will denote by $\res_K(\pi)$ the analogous restriction map after replacing $P$ by $P'$.
The map $\res_K(\pi)$ has then the property that for any $f\in \id_{P\cap K}^K(\pi)$, there exist a finite set $\{f_1,\ldots,f_k\}$ in $\id_{P'\cap K}^K(\pi\un)$ and rational functions $P_i\in R(T)$ such that for every unramified character $\chi$ of $\FF^\times$ we have 
\[\res_K(\ji(\pi)(\res_K^{-1}(f)\otimes\chi(\nu)))=\sum_iP_if_i,\] where we write $\chi(\nu)$ for the corresponding unramified character of $M$.

Since we assume $\pi$ to be admissible, there exists some $P\in R[T,T^{-1}]$ such that $P\cdot J_{P,P'}(\pi\un)$ preserves functions with entries in $\pi\otimes_R\xc_R[M]\otimes\psi\un\otimes_\nu R[T,T^{-1}]$.
We can evaluate 
$P\cdot J_{P,P'}(\pi\un)$ at $\chi(\nu)$ to obtain a morphism \[J_{P,P'}(\pi\otimes\chi(\nu))=J_{P,P'}^G(\pi\otimes\chi(\nu))\colon \ip(\pi\otimes\chi(\nu))\ra\id_{P'}^G(\pi\otimes\chi(\nu)).\] One can choose $P$ such that the above morphism does not vanish, and we will denote this $P$ by $P(\pi\otimes\chi(\nu))$, and this defines $J_{P,P'}(\pi\otimes\chi(\nu))$ uniquely up to a scalar in $R$. 
We let $\Lambda(\pi,P,P')$ be the order of the zero of $P(\pi)$ at $T=1$.  
If $\pi$ is irreducible, or a subquotient of a representation induced by irreducible representations, we denote the order of the pole of the associated $j$-function by $\alpha(\pi,P)$. Recall that the $j$-function $j(\pi,P)$ is given in this case by the scalar map $J_{\op,P}(\pi\un)\circ J_{P,\op}(\pi\un)$. Again composing with $\nu$ gives us a well defined notion of the order of a pole or zero.

We recall that $\alpha(\pi,P)$ only depends on the cuspidal support of $\pi$.
We also denote  \[\pad(\pi,P)\coloneqq \Lambda(\pi,P,\op)+\Lambda(\pi,\op,P)+\alpha(\pi,P).\]
The following is then easy to see.
\begin{lemma}\label{L:pder}
    For $\pi$ and $P$ as above we have $\pad(\pi,\op)=\pad(\pi,P)\ge 0$. Moreover, the following are equivalent.
    \begin{enumerate}
        \item $J_{P,\op}(\pi)$ is an isomorphism.
        \item $J_{\op,P}(\pi)$ is an isomorphism.
        \item $\pad(\pi,P)=0$.
    \end{enumerate}
\end{lemma}
Moreover, one can characterize precisely when $\Lambda(\pi,P,P')=0$, see for example the proof of \cite[Proposition IV.2.2]{waldspurger2003plancherel}, see also \cite[Lemma 4.7]{Droschl2025inter}. We will use here the language of \Cref{S:GL}.
\begin{lemma}\label{L:noresidue}
    Let $\pi$ and $P,P'$ be as above. Then $\Lambda(\pi,P,P')=0$ if and only if the map
    $r_{P'}(\ip(\pi))\ra \pi$ obtained by Frobenius reciprocity does not vanish on $r_{P'}(F(\oo')(\pi))\subseteq r_{P'}(\ip(\pi))$ if $\oo'\subsetneq\overline{\oo}$. 
\end{lemma}
Assume now that $P=P_\alpha$ and $P'=wPw^{-1}$ for some $w\in W_\alpha$. We let $s_1\ldots s_k=w$ be a reduced expression of $w$. 
Let $G_i$ be the Levi subgroup of  $(P\cup s_iPs_i) $ and $Q$ and $Q_i$ the parabolic subgroups of $G_i$ coming from $P$ and $s_iPs_i$.
The following is an immediate consequence of \cite[Proposition 7.8]{Dat05}.
\begin{lemma}\label{L:compositionJ}
    Assume that $J_{Q,Q_i}^{G_i}(\pi)$ is an isomorphism for $\ain{i}{1}{k}$. Then $J_{P,P'}(\pi)$ is an isomorphism and \[\Lambda(\pi,P,P')=\sum_{i=1}^k\Lambda(\pi,Q,Q_i).\]
\end{lemma}
We now state some explicit computations of \cite{Dat05} of the above morphisms in the case where $P$ is a maximal standard parabolic subgroup, $\pi$ is cuspidal and irreducible, and $P'$ is the opposite parabolic subgroup of $P$.
\begin{prop}[{\cite[Proposition 8.4]{Dat05}}]\label{P:polesgl}
    Assume that $G=G_n$ is a general linear group and $\pi=\rho_1\otimes\rho_2$ is cuspidal.
    If $\rho_1\ncong\rho_2\otimes\lvert\det\rvert^\epsilon,\,\epsilon\in\{-1,0,1\}$, then \[\Lambda(\pi,P,P')=\alpha(\pi,P)=0.\]
    If $\rho_1\cong \rho_2$, then 
    \[\Lambda(\pi,P,P')=1,\,\alpha(\pi,P)=-2.\]
Hence $\ji(\pi)$ is an isomorphism and $\ip(\pi)$ is irreducible.
    
    If $\rho_1\cong\rho_2\otimes\lvert\det\rvert^\epsilon,\,\epsilon\in\{-1,1\}$, we have that $\Lambda(\pi,P,P')=0$ and $\alpha(\pi,P)=1$.
\end{prop}
\begin{lemma}\label{L:shiftgl}
    Let $G=G_n$ and $\sigma=\sigma_1\otimes\ldots\otimes\sigma_k\in\irr_\ql(M_\alpha),\,\alpha=(\alpha_1,\ldots,\alpha_k)$ be a representation such that for $i\neq j$ and $\rho$ in the cuspidal support of $\sigma_i$, $\rho\lvert\det\rvert^\epsilon,\,\epsilon\in\{-1,0,1\}$, does not appear in the cuspidal support of $\sigma_j$. Then
    \[\sigma_1\times\ldots\times\sigma_k\]
    is irreducible and $\Lambda(\sigma,P_\alpha,\overline{P_\alpha})=\alpha(\sigma,P_\alpha)=0$.
\end{lemma}
\begin{proof}
    The condition on the cuspidal support together with the Geometric Lemma ensures that $\sigma$ appears with multiplicity $1$ in \[r_\alpha(\id_{P_\alpha}^G(\sigma))\] and \[r_{\overline{P_\alpha}}(\id_{\overline{P_\alpha}}^G(\sigma)).\] Thus \[\Lambda(\sigma,P_\alpha,\overline{P_\alpha})=\Lambda(\sigma,\overline{P_\alpha},P_\alpha)=0.\] Similarly, the condition on the cuspidal support also implies that $\alpha(\sigma,P_\alpha)=0$. By \Cref{L:compositionJ} and \Cref{L:pder} $J_{P_\alpha,\overline{P_\alpha}}(\sigma)$ is an isomorphism. It is well known that any irreducible subrepresentation of $\id_{\overline{P_\alpha}}^G(\sigma)$ is a quotient of $\id_{P_\alpha}^G(\sigma)$. Indeed, this is achieved via the involution $g\mapsto{}^tg^{-1}$ on $G_n$, \emph{cf}. \cite{Vig96}.
    Finally, we have by Frobenius reciprocity and the above observation that
\[\ho(\id_{P_\alpha}^G(\sigma),\id_{\overline{P_\alpha}}^G(\sigma))\]
    is one-dimensional. Thus the claim follows.
\end{proof}
We now state the analogue of \Cref{P:polesgl} for classical groups. We use here that $\ell$ is banal.
\begin{lemma}[{\cite[Proposition 8.4]{Dat05}}]\label{P:polesG}
    Assume that $G$ is a classical group and $\pi=\rho\otimes\sigma$ with $\rho$ a $\fc$-self-dual cuspidal representation of $G_m$.
    Then $\ji(\pi)$ is always an isomorphism.
    
    Firstly, the following are equivalent.
    \begin{enumerate}
        \item $\ji(\pi)$ is a scalar.
        \item $\Lambda(\pi,P,P')=1$.
        \item $\alpha(\pi,P)=-2$.
        \item $\ip(\pi)$ is irreducible.
    \end{enumerate}
    Secondly, the following are also equivalent.
    \begin{enumerate}
        \item $\ip(\pi)$ is reducible.
        \item $\ip(\pi)$ is semi-simple of length $2$.
        \item $\Lambda(\pi,P,P')=0$.
        \item $\alpha(\pi,P)=0$.
    \end{enumerate}
Finally, if $R=\ql$, $\rho$ unitary and $\alpha((\rho\otimes \chi)\otimes\sigma,P)\neq 0$ for some unramified character $\chi$, then
$\chi=\lvert\det\rvert^{\pm \frac{k}{2}}$, for some unique $k\in \ZZ_{\ge 0}$ with $\left\lvert\frac{k}{2}\right\rvert\le d_{W,m}.$
\end{lemma}
\begin{proof}
    The first two claims follow from \cite[Proposition 8.4]{Dat05}. The last claim follows immediately from the  main theorem of \cite{moeglin2002points} in the split case. For the others see \cite{Moeglin2002}, \cite{MoeglinTadic2002}.
\end{proof}
We will refer to the respective cases of whether $J_{P,P'}(\pi)$ is a scalar or not as \emph{type I} or \emph{type II}.
\begin{lemma}\label{L:shiftG}
    Let $G$ be a classical group and $\pi=\pi_1\otimes\ldots\otimes\pi_k\otimes \sigma\in\irr_\ql(M_\alpha)$, $\alpha=(\alpha_1,\ldots,\alpha_k)$ be a representation such that
    \begin{enumerate}
        \item for $i\neq j$ and $\rho$ in the cuspidal support of $\pi_i$, $\rho\otimes\lvert\det\rvert^\epsilon,\,\epsilon\in\{-1,0,1\}$ does not appear in the cuspidal support of $\pi_j$,
        \item for $\rho\in\irr_\ql(G_m)$ in the cuspidal support of $\pi_i$ such that there exists $k\in \ZZ$ with $\rho\otimes\lvert\det\rvert^\frac{k}{2}$ $\fc$-self-dual, we have $\lvert k\rvert> 2d_{W,m}$, where $W$ is such that $\sigma $ is a representation of $G(W)$.
    \end{enumerate} Then
    \[\pi_1\times\ldots\times\pi_k\rtimes\sigma\]
    is irreducible.
\end{lemma}
\begin{proof}
    The proof follows exactly the same line of reasoning as the one appearing in the proof of \Cref{L:shiftgl}, with the exception that we use the MVW-involution instead of the involution $g\mapsto {}^tg^{-1}$.
\end{proof}
\subsection{Reduction of intertwining operators}
We recall the character $\nu\colon\xc_R[M]\ra R[T,T^{-1}]$ from the construction of the intertwining operators. Note that it can be chosen for $R=\ql$ and $R=\fl$ in a compatible way.
For $\alpha\in\zl^\times$, we denote by $\nu_\alpha$ the unramified character of $M$ obtained as the composition of $\nu$ with the evaluation map $T\mapsto \alpha$.
\begin{lemma}\label{L:ratJ}
    Let $\tpi\in\rep_\ql(M)$, $\fo$ be an integral structure of $\tpi$, and write $\pi\coloneqq\fo\bl$. Assume that the following are satisfied.
    \begin{enumerate}
        \item $\df\ho(\ip(\pi),\id_{\op}^G(\pi))=1$.
        \item $\Lambda(\tpi,P,\op)=\Lambda(\pi,P,\op)=0$.
    \end{enumerate}
    Then \[J_{P,\overline{P}}(\tpi)(\ip(\fo))\subseteq \id_{\op}^G(\fo).\]
    Moreover, if for all $1\neq a\in\zl^\times$ with $a=1\mod\ell$, we have that \[\Lambda(\tpi\otimes\nu_a,P,\op)=\alpha(\tpi\otimes\nu_a,P)=0,\] then \[J_{P,\overline{P}}(\tpi)(\ip(\fo))\bl=\im(J_{P,\op}(\pi)).\]
\end{lemma}
\begin{proof}
We start by proving the first inclusion. Indeed, from the Brauer-Nesbitt principle of Vignéras, it follows that there exists $k\in\mathbb{N}$ such that the map $I\coloneqq \ell^kJ_{P,\overline{P}}(\tpi)$ satisfies the inclusion, and we fix a minimal such $k$. We want to show that $k=0$. We thus can reduce $I$ mod $\ell$ to obtain a non-zero map $I\bl\in \ho(\ip(\pi),\iop(\pi))$, which implies that up to a scalar $I\bl$ equals $J_{P,\overline{P}}(\pi)$. Note that by \Cref{L:noresidue} the latter does not vanish for some $f\in \ip(\pi)$ with support in $P\op$. Similarly, by \Cref{L:noresidue}, it follows that the same statement holds for $J_{P,\overline{P}}(\fo)$. Now if $k>0$, it would follow that $I$ vanishes on all $f\in \ip(\tpi)$ with support in $P\op$, a contradiction. Thus $k=0$ and the claim follows.

For the second claim we use  the rationality of intertwining operators. It suffices to prove that for $f\in \ip(\fo)$ with $f\mod\ell\neq 0$, $J_{P,\overline{P}}(\tpi)(f)=0\mod \ell$ implies $J_{P,\op}(\tpi)(f)=0$. 
Let $K'$ be a sufficiently small open compact subgroup of $G$ fixing $f$.
We therefore find a basis $\{h_1,\ldots,h_k\}$ of $\id_{P\cap K}^K(\tpi)^{K'}$  and $\{f_1,\ldots,f_k\}$ of $\id_{\op\cap K}^K(\tpi)^{K'}$
such that there exists a $k\times k$ matrix $P=(P_{i,j})$ with entries in $R[T,T^{-1}]$ such that \[\res_K(J_{P,\op}(\tpi\otimes\nu_a)(\res_K^{-1}(h_j)))=\sum_iP_{i,j}(a)f_i.\]
We can assume without loss of generality that the $h_i$'s and $f_i$'s have values in $\fo$ and hence for every $\zl$-valued unramified character $\nu_a$ we have that $P_{i,j}(a)\in\zl$.
We let \[\nu\colon \xc_\ql[M]\ra \ql[T,T^{-1}]\] be the ring map used in the definition of intertwining operators.
 By assumption, the matrix $P(1)\mod \ell$ has non-maximal rank and hence $\det(P(1))=0\mod\ell$.
 Let $\alpha_1,\ldots,\alpha_m$ be the zeros of $\det(P)$ over $\ql$. Since $\det(P(1))\in \zl$ and equals $0$ mod $\ell$, at least one of the $\alpha_i$  must be integral and reduce to $1\mod \ell$.
 Therefore, there exists an integral character $\nu_{\alpha_i} $, which lifts the trivial character over $\fl$ to $\zl$, such that $\det(P(\alpha_i))=0$. But then this implies that $J_{P,\op}(\tpi\otimes\nu_{\alpha_i})$ is not an
  isomorphism, which in turn implies via \Cref{L:pder} and the assumption of the lemma that $\alpha_i=1$.
  Hence $J_{P,\op}(\tpi)(f)=0$ as desired. 
\end{proof}
If we assume that $P$ is a maximal standard parabolic subgroup of $G$, the character $\nu$ can be taken to send $(a,b)\in M$ to $T^{-\mathrm{val}_\FF(\det(a))}$. In this case, the condition of the lemma is easy to verify in the following case.
Assume that both $\tpi$ and $\pi$ are irreducible, $\pi$ appears in $r_{\op}(\ip(\pi))$ with multiplicity $1$, and recall that $\ell$ is banal. Let $a\in\zl$ with $a=1\mod \ell$ and consider the twist $\tpi\otimes\nu_a$. By \cite[Lemma 7.2]{Dat05}, it follows that $\Lambda(\tpi\otimes\nu_a,P,\op)=0$. It thus suffices to show that $\alpha(\tpi\otimes\nu_a,P)=0$ for all $a\neq 1$ in order to apply the lemma.
\section{Progenerators}
We assume throughout this section that $G$ is a classical group and $\ell$ is a banal prime with respect to $G$.
Let $P=P_\alpha\subseteq G,\alpha=(\alpha_1,\ldots,\alpha_k)$, and $\pi=\rho_1\otimes\ldots\otimes\rho_k\otimes\sigma\in\irr_{R,c}(M_\alpha)$ where all $\rho_i$ are $\fc$-self-dual representations of $G_{\alpha_i}$ and $\sigma $ is a cuspidal representation of a classical group. We write $M=M_\alpha$. 
We let $\Pi\coloneqq \ip(\pi\otimes_R \xc_R[M]\otimes\psi\un)$.
Let $\rep_{R,\pi,M}(G)$ be the block of $\rep_R(G)$ associated to the cuspidal representation $\pi$ of $M$.
\begin{theorem}[{\cite[1.6]{Roche_2002}}]\label{T:progenerator}
    The representation $\Pi$ is projective and faithful in $\rep_{R,\pi,M}(G)$.
    Hence $\rep_{R,\pi,M}(G)$ is equivalent to the category of $\mathrm{End}(\Pi)$-modules.
\end{theorem}
We note that the theorem is stated in \cite[1.6]{Roche_2002} only for representations of characteristic $0$. However, the result requires only two ingredients: Firstly, that every cuspidal representation is projective, which is satisfied since $\ell$ is banal, and secondly, the second adjointness theorem, which has been established in \cite[Corollary 1.3]{DatHelmKurinczukMoss_2024_2} in the desired generality.

As a corollary we obtain that any endomorphism of $\ip(\pi)$ lifts to an endomorphism of $\Pi$. The endomorphism algebra of $\Pi$ was described in \cite{Heiermann2011}. To do this we need to recall some of the notation of \cite{Heiermann2011}, which he only employs for split symplectic or orthogonal classical groups, but the results we need readily generalize to our context. 
Firstly, let $\scrO=\{\pi\otimes \chi:\chi\in\xc_R[M]\}$ and let $W_\alpha(\scrO)$ be a subgroup of $W_\alpha$ that stabilizes $\scrO$. For each $w\in W_\alpha(\scrO)$ one has that ${}^w\pi\cong \pi$, where we denote by ${}^w$ the twist by the element $w$. One can then fix canonical isomorphisms $\rho_w\colon \id_{wPw^{-1}}(\pi)\ra\ip(\pi)$ for all $w\in W_\alpha(\scrO)$, which are defined as compositions of normalized intertwining operators, see \cite[2.4, 2.5]{Heiermann2011}.
Using the rational intertwining operators of \cite{Dat05} and the above $\rho_w$, one can then finally describe for each $w\in W_\alpha(\scrO)$ an element \[A_w(\pi)\in\ho(\ip(\pi\otimes_R\xc_R[M]\otimes\psi\un),\ip(\pi\otimes_R\cK\otimes\psi\un)).\] 
After multiplying by a suitable element in $\cK$, see \cite[4.6]{Heiermann2011}, one obtains elements $J_w(\pi)$ which are defined as compositions of suitable morphisms $A_{w'}(\pi)$. We then have elements \[J_w(\pi)\in \ho(\ip(\pi\otimes_R\xc_R[M]\otimes\psi\un),\ip(\pi\otimes_R\cK\otimes\psi\un))\] for all $w\in W_\alpha(\scrO)$.
\begin{theorem}[{\cite[Theorem 4.9]{Heiermann2011}}]\label{T:Heier}
    There is an isomorphism of $\cK$-vector spaces
    \[\ho(\ip(\pi\otimes_R\xc_R[M]\otimes\psi\un),\ip(\pi\otimes_R\cK\otimes\psi\un))\cong  \bigoplus_{w\in W_\alpha(\scrO)}\cK\cdot J_w(\pi).\]
\end{theorem}
Let us remark quickly that although \cite{Heiermann2011} proves the theorem only for complex representations, the proof goes through \emph{mutatis mutandis} in our banal setting, since both the Geometric Lemma and the construction of the intertwining operators work in exactly the same fashion.

We return now to $\ed(\ip(\pi))$ and define for each element $w\in W_\alpha(\scrO)$ an intertwining operator $J'_{w}(\pi)\colon \ip(\pi)\ra\id_{wPw^{-1}}^G(\pi)\ra\ip(\pi)$, where the second map is the isomorphism coming from conjugation by $w$.
As a corollary of \Cref{T:Heier} we obtain the following.
\begin{corollary}\label{C:pro}
    The set $\{J'_{w}(\pi):w\in W_\alpha(\scrO)\}$ is an $R$-generating set of $\ed(\ip(\pi))$.
\end{corollary}
\begin{proof}
    Indeed, by \Cref{T:progenerator} any map in $\ed(\ip(\pi))$ can be lifted to a map in \[\ho(\ip(\pi\otimes_R\xc_R[M]\otimes\psi\un),\ip(\pi\otimes_R\xc_R[M]\otimes\psi\un)).\] By the explicit description of the latter space in \Cref{T:Heier}, \cite[Theorem 5.10]{Heiermann2011} and \cite[Proposition 7.8]{Dat05} the claim follows.
\end{proof}
\begin{prop}\label{P:basecasetemperedl}
    Let $P=P_\alpha\subseteq G,\,\alpha=(\alpha_1,\ldots,\alpha_m)$, be a parabolic subgroup of a classical group.
    Let $\rho_i\in\irr_{R,c}(G_{n_i}),\,\ain{i}{1}{k}$, be $\fc$-self-dual and pairwise non-isomorphic. Let  $\sigma\in \irr_{R,c}(G')$, where $G'$ is the classical component of $M_\alpha$. Let $d$ be the number of $\rho_i$ such that $\rho_i\otimes \sigma$ is of type II.
    Then for all $r_i\in\NN,\,\ain{i}{1}{k}$, such that
    \[\alpha=(\underbrace{n_1,\ldots,n_1}_{r_1},\ldots,\underbrace{n_k,\ldots,n_k}_{r_k}),\qquad m=\sum_{i=1}^k r_i,\]
    we have that
    \[\Pi_r\coloneqq \ip(\pi),\,\pi\coloneqq\overbrace{\rho_1\otimes\ldots\otimes\rho_1}^{r_1}\otimes\ldots\otimes\overbrace{\rho_k\otimes\ldots\otimes\rho_k}^{r_k}\otimes \sigma\] is semi-simple of length $2^d$ and multiplicity-free.

In particular, if $R=\fl$ and we have a lift $\tpi$ of $\pi$ to $\ql$, any irreducible summand of $\Pi_r$ lifts uniquely to a summand of $\ip(\tpi)$.    
\end{prop}
\begin{proof}
    Over $\ql$ the result follows from \cite{arthur2013endoscopic}, see also \cite[§4]{atobe2022soclesclassical}.
    Over $\fl$, we use the case of $\ql$ as an input. 
    Firstly, by \Cref{L:shiftgl} we can assume without loss of generality that the $\rho_i$ such that $\rho_i\otimes \sigma$ is of type II are exactly those with $1\le i\le d$. 
    We show first how \Cref{C:pro}, \Cref{P:polesG}, \Cref{P:polesgl}, and \Cref{L:compositionJ} imply \[\dim_\fl\ed(\Pi_r)= 2^d.\] Indeed, denote for each $\ain{i}{1}{d}$ by $s_i$ the simple element of $W_\alpha$ such that $s_i$ acts on the Levi factor
    \[M_\alpha=G_{n_1}^{r_1}\times\ldots\times G_{n_k}^{r_k}\times G'\] trivially everywhere except on the $r_i$-th copy of $G_{n_i}$ in the $i$-th block, where it acts via $g\mapsto {}^tg^{-1}$.
     We then have nontrivial morphisms \[J'_{w}(\pi)\in\ed(\Pi_r),\, w=s_1^{\epsilon_1}\ldots s_d^{\epsilon_d},\, \epsilon_i\in\{0,1\}.\]
     Moreover, the poles of these intertwining operators vanish. Indeed by \cite[Proposition 7.8]{Dat05} and \Cref{L:compositionJ}, it suffices to show that \[\Lambda(\rho_i\otimes \sigma,P,\overline{P})=0\] for all $\rho_i$ of type II. But this is nothing but \Cref{P:polesG}.
By \Cref{C:pro} we have the lower bound \[\dim_\fl\ed(\Pi_r)\ge 2^d.\] To see that this is an equality, we note that by \Cref{L:compositionJ} any other intertwining operator is a composition of the above $J'_{w}(\pi)$ together with intertwining operators induced from the intertwining operators discussed in the first case of \Cref{P:polesgl} and \Cref{P:polesG}. By \Cref{C:pro} it thus follows that \[\dim_\fl\ed(\Pi_r)\le 2^d.\]
Furthermore, denote for $\pi\in \irr_\fl(G)$ by $d_\pi$ the multiplicity of $\pi$ in the socle of $\Pi_r$, and hence also by the MVW-involution in the cosocle of $\Pi_r$. Thus, we obtain \[\sum_\pi d_\pi^2\le 2^d\] by considering maps of the form $\Pi_r\ra\pi\ra\Pi_r$. Note that if $\Pi_r$ were not semi-simple, then the identity would not be within the span of these morphisms and hence the inequality would be strict.

    But now we can choose a lift $\tpi$ of $\pi$ thanks to \Cref{T:cuslift} and use the fact that we already proved that $\ip(\tpi)$ is semi-simple of the desired length. In particular, we know that the socle of $\Pi_r$ has length at least $2^d$. From this it quickly follows that the socle is multiplicity free and $\sum_\pi d_\pi^2=2^d$, hence $\Pi_r$ is semi-simple and multiplicity free of length $2^d$.
\end{proof}
\section{Derivatives}\label{S:der}
 In this section we recall the derivatives of general linear and classical groups, following the appendix of \cite{atobe2024local}. They were first introduced by \cite{Min08} and \cite{Jader}. For the modular case see also \cite{Dro24}. Note that in the banal case cuspidal representations are always projective in their respective block. Therefore, the same definitions as in \cite{atobe2024local} work without adjustments.
 
  Assume $G=G_n$ is a general linear group and let $\rho\in\irr_{R,c}(G_m)$. A representation $\pi\in\irr_R(G)$ is called $\rho$-reduced if either $m>n$ or $r_{m,n-m}(\pi)$ does not contain a $(G_m,\rho)$-isotypic subquotient.
    \begin{lemma}\label{L:dergl}
        Assume $G=G_n$ is a general linear group and let $\rho\in\irr_{R,c}(G_m)$. Let $d_\rho(\pi)$ be the maximal $k\in\NN$ such that $r_{km,n-km}(\pi)$ contains a subquotient of the form $\rho^k\otimes \pi'$. Put $k=d_\rho(\pi)$. Then the following holds.
        \begin{enumerate}
            \item $\dr(\pi)\coloneqq \pi'$ is uniquely determined by $\pi$.
            \item $\dr(\pi)$ is $\rho$-reduced.
            \item $\Lambda(\dr(\pi)\otimes\rho^k,P_{n-km,km},\overline{P_{n-mk,km}})=0$.
            \item The image of $J_{P_{n-km,km},\overline{P_{n-mk,km}}}(\dr(\pi)\otimes \rho^k)$ is isomorphic to $\pi$.
            \item The integer $d_\rho(\pi)$ is the maximal $k$ such that there exists an irreducible $\pi'$ with $\pi'\times\rho^k$ admitting $\pi$ as a quotient.
            \item If $\pi'\times\rho^{d_\rho(\pi)}$ admits $\pi$ as a quotient, $\pi'\cong \dr(\pi)$.
        \end{enumerate}
    \end{lemma}
Assume now that $G$ is a classical group and let $\rho\in\irr_{R,c}(G_m)$. A representation $\pi\in\irr_R(G)$ is called $\rho$-reduced if either $m>d_G$ or $r_{m}(\pi)$ does not contain a $(G_m,\rho)$-isotypic subquotient.
\begin{lemma}\label{L:dercl}
   Assume $G$ is a classical group and let $\rho\in\irr_{R,c}(G_m)$ be non-$\fc$-self-dual. Let $d_\rho(\pi)$ be the maximal $k\in\NN$ such that $km\le d_G$ and $r_{km}(\pi)$ contains a subquotient of the form $\rho^k\otimes \pi'$. Put $k=d_\rho(\pi)$. Then the following holds.
        \begin{enumerate}
            \item $\dr(\pi)\coloneqq \pi'$ is uniquely determined by $\pi$.
            \item $\dr(\pi)$ is $\rho$-reduced.
            \item $\Lambda(({}^\fc\rho^\lor)^k\otimes \dr(\pi),P_{km},\overline{P_{km}})=0$.
            \item The image of $J_{P_{km},\overline{P_{km}}}(({}^\fc\rho^\lor)^k\otimes \dr(\pi))$ is isomorphic to $\pi$.
            \item The integer $d_\rho(\pi)$ is the maximal $k$ such that there exists an irreducible $\pi'$ with $({}^\fc\rho^\lor)^k\rtimes \pi'$ admitting $\pi$ as a quotient.
            \item If $({}^\fc\rho^\lor)^{d_\rho(\pi)}\rtimes \pi'$ admits $\pi$ as a quotient, $\pi'\cong \dr(\pi)$.
        \end{enumerate}
    \end{lemma}\section{Lifting representations}
In this section we show the promised lifting properties. Throughout this section we assume that $\ell$ is a banal prime.
We start with the case of the general linear group, which was already proven in \cite{minguez2013representations}. However, our methods allow us to give a new proof.    
\begin{corollary}[{\cite[Theorem 6.1]{minguez2013representations}}]\label{T:lift}
    Let $\tpi\in\irr_\ql(G_n)$ be an irreducible representation with an integral banal cuspidal support. Then $\rl(\tpi)$ is irreducible. 
    If $\pi\in\irr_\fl(G_n)$, then $\pi$ admits a lift to $\ql$.
\end{corollary}
\begin{proof}
    We argue by induction on $n$, the base case being \Cref{T:cuslift}. Let $\trho$ be a cuspidal representation such that $d_\trho>0$ and put $\rho\coloneqq \rl([\trho])$. If $\mathcal{D}_\trho(\tpi)=0$, the case follows from the fact that $\rho\times\ldots\times\rho$ is irreducible, \emph{cf.} \cite{minguez2013representations}. Otherwise, we know that $\mathcal{D}_\trho(\tpi)$ and $\rl([\mathcal{D}_\trho(\tpi)])$, which we know by induction to be irreducible, are $\trho$-reduced and $\rho$-reduced, respectively. Indeed, the $\ql$-claim follows by definition and the $\fl$-claim follows from the first as follows. Recall that the Jacquet-functor commutes with reduction mod $\ell$ up to semi-simplification and that the cuspidal support is banal. It follows from the Geometric Lemma that any $\rho^{d_\trho}$-equivariant part of the Jacquet-functor of $\rl(\tpi)$ needs to be contained in 
    \[\rl(\trho^{d_\trho})\otimes \rl(\mathcal{D}_{\trho}(\tpi)).\]
    By the induction hypothesis, this is an irreducible representation.
    
    The first part of the corollary follows now from the induction hypothesis, then \Cref{L:ratJ} and finally \Cref{L:shiftgl}. 
    Indeed, according to the comment after \Cref{L:ratJ}
    one just needs to check that $\alpha(((\trho)^{d_\trho}\otimes\nu_a)\otimes \mathcal{D}_\trho(\tpi))=0$ for $a\neq 1$. By the assumption on the cuspidal support and \Cref{L:shiftgl} the claim follows.
    The last part follows from \Cref{L:cuspsupext}.
\end{proof}
\begin{rem}
    Write $\pi=\rl(\tpi)$. Note that in the above corollary also the $\Z$-parameters of \cite{MinSec14} of $\tpi$ and $\pi$ agree thanks to the explicit formulas for derivatives in \cite{Dro24}.
    Of course this was already shown in \cite{minguez2013representations}. Nonetheless, we obtain a new proof of that fact.
\end{rem}
\begin{lemma}\label{L:basecasetempered}
    Let $\tpi\in\irr_\ql(G)$ be a tempered representation with $G$ a classical group. If $\tpi$ is reduced for all non-$\fc$-self-dual cuspidal representations of an arbitrary general linear group then there exist $P_\alpha\subseteq G$ and a $\fc$-self-dual cuspidal representation $\trho_1\otimes\ldots\otimes\trho_k\otimes\sigma$ of $M_\alpha$ such that  \[\tpi\hra \trho_1\times\ldots\times\trho_k\rtimes \sigma.\]
\end{lemma}
\begin{proof}
    In the case of symplectic and split orthogonal groups this has been shown in {\cite[5.3]{AtobeMinguez2020}}. We give a general proof using Casselman's criterion and argue by induction on the rank of $G$. Indeed, Casselman's criterion gives $\tpi\hra\trho\rtimes\tpi'$, where $\tpi'$ is a tempered representation and $\trho$ is a suitable cuspidal representation. But by the assumption on $\tpi$, $\trho$ is $\fc$-self-dual. It suffices now to check by the induction hypothesis that $\tpi'$ satisfies the conditions of the lemma. If it admitted a non-zero derivative with respect to a non-$\fc$-self-dual cuspidal representation $\trho'$, we could embed $\tpi\hra\trho\times\trho'\rtimes\sigma$ for a suitable irreducible representation $\sigma$. If $\trho\times\trho'$ is irreducible, $\trho\times\trho'\cong\trho'\times\trho$ and we have a contradiction to the assumptions on $\tpi$ by applying Frobenius reciprocity to the inclusion
    \[\tpi\hra\trho'\times\trho\rtimes\sigma.\]Thus $\trho'\cong\trho\otimes\lvert\det\rvert^{\pm 1}$ by \Cref{P:polesgl}. By temperedness of $\tpi'$ and Casselman's criterion, we have that only the case of $+1$ can occur. But the unique essentially tempered subquotient $\Z([0,1]_\trho)$ of $\trho\times\trho\otimes\lvert\det\rvert$ is a subrepresentation of $\trho\otimes\lvert\det\rvert\times\trho$.
    Moreover, Casselman's criterion implies that among the two subquotients $\Z([0,1]_\trho)$ and $\Z([0]_\trho+[1,1]_\trho)$ of $\trho\times\trho\otimes\lvert\det\rvert$ only the former can appear in the Jacquet-module of $\tpi$. It is then easy to see that Frobenius reciprocity implies that
    \[\tpi\hra \Z([0,1]_\trho)\rtimes \sigma\]
    and hence $\tpi$ admits a non-zero derivative with respect to $\trho\otimes\lvert\det\rvert$, a contradiction.
\end{proof}
\begin{corollary}\label{C:lifting}
    Let $\tpi\in\irr_\ql(G)$ be an irreducible representation with an integral banal cuspidal support. Then $\rl(\tpi)$ is irreducible.
\end{corollary}
\begin{proof}
    We argue by induction on the rank of $G$ to prove the claim.
    
    If $\tpi$ is tempered, we treat first the case where $\tpi$ admits a non-zero derivative with respect to some non-$\fc$-self-dual cuspidal representation $\trho$. In this case we can argue precisely as in \Cref{T:lift}, using \Cref{L:dercl} instead of \Cref{L:dergl}.
    
    If $\tpi$ only admits derivatives with respect to $\fc$-self-dual cuspidal representations, we can apply \Cref{L:basecasetempered} to obtain a $\fc$-self-dual cuspidal representation $\trho_1\otimes\ldots\otimes\trho_k\otimes\sigma$ of $M_\alpha$ such that  \[\tpi\hra \trho_1\times\ldots\times\trho_k\rtimes \sigma.\] In this case we can apply \Cref{P:basecasetemperedl} to finish the tempered case.

    We now assume that $\tpi$ is not tempered and denote by $(P_\alpha,\tchi,\tsigma)$, $\alpha=(\alpha_1,\allowbreak\ldots,\allowbreak\alpha_k)$ a Langlands datum of $\tpi$. We already know that $\tsigma\otimes\tchi$ reduces to an irreducible representation $\sigma\otimes\chi$ mod $\ell$ thanks to \Cref{T:lift} and the induction hypothesis. Since the cuspidal support of $\tpi$ is banal, one can thus prove exactly as in \cite{Borel_Wallach_2000} that $\sigma\otimes\chi$ appears with multiplicity $1$ in $r_{\overline{P_\alpha}}(\id_{P_\alpha}^G(\sigma\otimes \chi))$. 
    Moreover, we recall from \cite{Borel_Wallach_2000} that firstly, there is a morphism, unique up to a scalar,
    \[\id_{P_\alpha}^G(\tsigma\otimes \tchi)\ra \id_{\overline{P_\alpha}}^G(\tsigma\otimes \tchi),\] which is secondly given by $J_{P_\alpha,\overline{P_\alpha}}(\tsigma\otimes\tchi)$ and thirdly, satisfies \[\Lambda(\tsigma\otimes\tchi,P_\alpha,\overline{P_\alpha})=0.\] All three statements are a direct consequence of the fact that $\tsigma\otimes\tchi$ appears with multiplicity $1$ in $r_{\overline{P_\alpha}}(\id_{P_\alpha}^G(\tsigma\otimes \tchi))$. Moreover, the image of this morphism is isomorphic to $\tpi$.

    By \cite[Proposition 7.8]{Dat05} we can describe this morphism as follows. Write $\tchi=\tchi_1\otimes\tchi'$, $\tsigma=\tsigma_1\otimes\tsigma'$, $\alpha'=(\alpha_2,\ldots,\alpha_k)$, $M_{\alpha_1}=G_{\alpha_1}\times G'$, where $\tchi_1$ and $\tsigma_1$ denote the contribution of the first $\mathrm{GL}_{\alpha_1}$-factor of the Levi subgroup. Then, up to a nonzero scalar in $\zl^\times$, we have that
    \[
\begin{gathered}
J_{P_\alpha,\overline{P_\alpha}}(\tsigma\otimes\tchi)\\
=J_{P_{\alpha_1},\overline{P_{\alpha_1}}}(\tsigma_1\otimes\tchi_1\otimes \id_{\overline{P_{\alpha'}}}^{G'}(\tsigma'\otimes\tchi'))\circ\mathbbm{1}_{\tsigma_1\otimes\tchi_1}\rtimes J_{P_{\alpha'},\overline{P_{\alpha'}}}(\tsigma'\otimes\tchi')
\end{gathered}
\]
and all involved intertwining operators have no pole. We now argue by induction on $k$ that $J_{P_\alpha,\overline{P_\alpha}}(\tsigma\otimes\tchi)$ preserves the integral structures as in \Cref{L:ratJ}. In the case $k=1$, this is nothing but \Cref{L:ratJ}, and \Cref{L:ratJ} gives also the preservation of the integral structure of the first morphism in the above composition, whereas the induction hypothesis on $k$ gives us that the second morphism in the composition preserves the integral structure. The claim thus follows from the fact that $\tpi$ is the image of the intertwining operator and \Cref{L:ratJ}.
\end{proof}
\begin{corollary}\label{C:main}
    Let $\pi$ be a banal representation of a classical group of symplectic, orthogonal, or unitary type. Then $\pi$ admits a lift to $\ql$. Moreover, if the group is one of the unramified groups $G_n,\,\Sp_n,\, \Oo_n$ or $\tu_n$, every irreducible representation over $\fl$ is banal if $\ell$ is banal.
\end{corollary}
We finish with the following proposition.
\begin{prop}
    Let $\pi\in\irr_{\fl}(G)$ and $\ell$ banal with respect to $G$.
    If there exists an essentially discrete series lift to $\ql$, then any lift to $\ql$ is an essentially discrete series.
\end{prop}
\begin{proof}
     Let $\tpi_1$ and $\tpi_2$ be two lifts of $\pi$, with $\tpi_1$ an essentially discrete series. We want to show that $\tpi_2$ is an essentially discrete series.

    If $G$ is a general linear group, essentially discrete series representations are classified in \cite{BerZel77}. Namely, let $\trho$ be a cuspidal representation, and $a\le b\in\ZZ$. Set then $\Z([a,b]_\trho)$ to be the unique irreducible subrepresentation of \[\trho\otimes\lvert\det\rvert^b\times\ldots\times\trho\otimes\lvert\det\rvert^a.\] Then any essentially discrete series representation is of this form. In particular, we assume that $\tpi_1=\Z([a,b]_\trho)$ is of this form.
    
    Since $\ell$ is banal, it follows that $\pi$ has to be the unique irreducible subrepresentation of $\rho\otimes\lvert\det\rvert^b\times\ldots\times\rho\otimes\lvert\det\rvert^a$, $\rho=\rl([\trho])$, which we also denote by $\Z([a,b]_\rho)$, see for example \cite{MinSec14}. Moreover, $\Z([a,b]_\trho)$ is a lift of $\Z([a,b]_\rho)$. We now prove by induction on the rank of $G$ that $\tpi_2$ is of the form $\Z([a,b]_{\trho'})$, where $\trho'$ is a lift of $\rho$.
    Indeed, if $\tpi_2$ is any other lift of $\pi$, then there exist $\trho''$ and $\tpi'$ irreducible such that $\tpi_2\hra\trho''\times \tpi'$. It follows quickly that $\rl([\trho''])=\rho\otimes\lvert\det\rvert^b$ and $\rl([\tpi'])=\Z([a,b-1]_\rho)$. By the induction hypothesis, $\tpi'\cong\Z([a,b-1]_{\trho'})$ for some lift $\trho'$ of $\rho$. Now $\tpi_2\ncong \trho''\times \tpi'$, since otherwise we get via the Geometric Lemma a contradiction to \cite[Lemma 7.13]{MinSec14}. Thus it follows by \cite[Theorem 7.23]{MinSec14} that $\trho''\cong\trho'\otimes\lvert\det\rvert^\epsilon,\,\epsilon\in\{a-1,b\}$. Again by \cite[Lemma 7.13]{MinSec14} and since $\ell$ is banal, $\epsilon=b$ and hence the claim follows from \cite[Proposition 7.16]{MinSec14}.

    Next, we come to the case of classical groups. Note that in the general linear case we proved that every lift of $\pi$ had banal cuspidal support and we will show that the same remains true in the classical case. We argue by induction on the rank of $G$.
    
    Let $P_\alpha$ be a maximal parabolic subgroup and $\Pi$ an irreducible constituent of $r_\alpha(\tpi_2)$. Then by the above observation we have that $\Pi\cong \Z([a,b]_{\trho_2})\otimes \tsigma_2,$ where $\trho_2\in\irr_{\ql,c}(G_m)$. Indeed, this follows immediately from Casselman's criterion. Let $\tchi_2$ be the central character of $\trho_2$. We want to apply Casselman's criterion to conclude that $\tpi_2$ is an essentially discrete series, and for this it suffices to show that \[\frac{a+b}{2}+\mathrm{Re}(\tchi_2)>0\]
    for all maximal parabolic subgroups $P_\alpha$ and irreducible subquotients $\Pi$ of the corresponding Jacquet-module. By contradiction, we assume the contrary and it suffices to show the claim where $\Pi$ is a quotient of the Jacquet-module. Indeed, the $\mathrm{Ext}^1$-groups between representations with different central characters vanish.
    
    By Frobenius reciprocity we have that 
    \[\tpi_2\hra \Z([a,b]_{\trho_2})\rtimes \tsigma_2.\]
    We set $\sigma\coloneqq \rl([\tsigma_2])$.
    By \cite[Theorem 3.2]{Droschl2025symp} and the Aubert-Zelevinsky involution, we also note that $\sigma$ is uniquely determined by $\pi$ as the irreducible representation such that \[\pi\hra  \Z([a,b]_{\rho_2})\rtimes\sigma,\, \rho_2\coloneqq \rl([\trho_2]).\]
    Note that the argument in \cite{Droschl2025symp} is only written for symplectic groups. However, the exact same proof works for all classical groups thanks to the MVW-involution.
    We can also find $\trho_1$, $a',b'\in\ZZ$ and $\tsigma_1$ such that 
   $\Z([a',b']_{\trho_1})\otimes \tsigma_1$ appears as a quotient of $r_\alpha(\tpi_1)$, and $\Z([a,b]_{\trho_2})$ and $\Z([a',b']_{\trho_1})$ reduce to the same representation mod $\ell$. Let $\tchi_1$ be the central character of $\trho_1$. By Casselman's criterion $\tsigma_1$ is an essentially discrete series representation, and by the induction hypothesis it has banal cuspidal support, hence reduces mod $\ell$ to an irreducible representation, which by Frobenius reciprocity and the above fact has to be isomorphic to $\sigma$. 
   
We will now quickly argue that we cannot have an isomorphism \[\tpi_2\cong \Z([a,b]_{\trho_2})\rtimes\tsigma_2.\] Indeed, it would imply that there exists an isomorphism \[\tpi_1\cong \Z([a',b']_{\trho_1})\rtimes \tsigma_1.\]
To see this, note that since $\tpi_1\hra \Z([a',b']_{\trho_1})\rtimes \tsigma_1$, it suffices to show that their semi-simplifications agree. To see this, it suffices to compare the Jacquet-modules for some non-trivial maximal parabolic subgroup, since the banality of $\ell$ implies that there cannot exist a cuspidal non-supercuspidal representation as a subquotient of $\tpi_1$. But to show that the Jacquet-modules have the same length, it suffices to show that reduction $\mod\ell$ has the same length. This follows from the claim regarding the shape of $\tpi_2$.
But the isomorphism \[\tpi_1\cong \Z([a',b']_{\trho_1})\rtimes \tsigma_1\]would contradict via the Geometric Lemma and Casselman's criterion the discreteness of $\tpi_1$.
    
We will now leverage the above non-existence of an isomorphism to finish the proof. By the induction hypothesis we know that the cuspidal support of $\tsigma_2$ is banal.
   Thus ${\trho_2}\otimes\lvert\det\rvert^k$ is $\fc$-self-dual for some $k\in\frac{1}{2}\mathbb{Z}$. Indeed, if $k\notin\frac{1}{2}\mathbb{Z}$, we can apply \Cref{L:shiftG}, and obtain the above excluded isomorphism.
   Next, we will argue that $\frac{a'+b'}{2}+\mathrm{Re}(\tchi_1)=\frac{a+b}{2}+\mathrm{Re}(\tchi_2)$, which will finish the proof. To see this, we note that since $\ell$ is banal, we would have otherwise that there exists $k\in\ZZ$ such that $\frac{a+b}{2}+\mathrm{Re}(\tchi_2)=\frac{a'+b'}{2}+\mathrm{Re}(\tchi_1)+\frac{k}{2}\ge d_{W,m(b-a+1)}$ and hence we can apply again \Cref{L:shiftG}, and obtain the excluded isomorphism.
\end{proof}
\begin{rem}
    Note that the above proposition is no longer true if one replaces \emph{discrete series} with \emph{tempered}.
    For example, let $\rho\otimes\sigma$ be a cuspidal $\fl$-representation of the Levi-factor $M_\alpha$ of a maximal parabolic subgroup of a classical group such that $\rho\rtimes \sigma$ is irreducible. Then any lift $\trho\rtimes\tsigma$ with $\trho$ $\fc$-self-dual is also tempered. However, the lifts $\trho\otimes\lvert\det\rvert^k\rtimes\tsigma$ are irreducible for $k\in\NN$ with $k\gg0$ and hence provide examples of non-tempered lifts when $o(q)\mid k$.
\end{rem}
\section{Applications to the modular theta correspondence}\label{S:Howe}
In this section we give a straightforward application of the above results to the modular theta correspondence. In the classical setting over $\CC$ the precise behavior of Theta lifts is at this point well understood, and we refer the reader to \cite{GanAWS2022} for an excellent exposition of these topics for example. In the modular setting in type II \cite{Min08} gives a satisfactory answer in the banal case, which was extended in \cite{Droschl2026} to any $\ell$ which does not divide $q$. Note that as soon as one leaves the banal setting, classical Howe duality breaks down and several new phenomena appear. However, in the case of dual pairs of type I, not even the banal case has been treated, due to some difficulties involving the composition series of degenerate principal series. We are now able to overcome these issues and prove a modular Howe-duality in the (strongly) banal case in this section.

We recall the setup of the (modular) Theta correspondence as presented in \cite{Trias2025}. Throughout this section we assume that $\FF$ has characteristic $0$.
Let $\epsilon\in\{\pm1\}$, $W$ be an $\epsilon$-hermitian $\EE$-vector space of dimension $n$ and $V$ a $-\epsilon$-hermitian $\EE$-vector space of dimension $m$. If $\EE=\FF$ we also assume that both $n$ and $m$ are even, and set $\epsilon_0$ to $\epsilon$ if $\EE=\FF$ and to $0$ otherwise. The parity condition on $n$ and $m$ exists only to avoid metaplectic covers, since we do not prove the analogous results for them.
 We also assume without loss of generality that $n\le m+\epsilon_0$.
We denote their symmetry groups with $G(W)$ and $G(V)$. Moreover, let $\WW=W\otimes V$ be equipped with its natural symplectic form. Then $G(W)\times G(V)$ is a type I dual pair inside $G(\WW)$. We fix an $\zl$-valued smooth additive character $\psi\colon \FF\ra \zl$ and denote its reduction mod $\ell$ by abuse of notation with the same letter. 

We denote by $\omega_{\psi,R}$ the Weil representation of $G(\WW)$ over $R$ and for $\pi\in\irr_R(G(W))$ we let $\Theta_{W,V,\psi,R}(\pi)$ denote the big theta lift of $\pi$.
Denote by $W^-$ the same space as $W$, but whose form has been multiplied by $-1$. Recall that part of the construction of $\omega_{\psi,R}$ is a choice of splitting characters $\chi_W$ and $\chi_V$, which are unitary $\zl$-characters of $\FF^\times$.
By abuse of notation, we will also denote their reductions mod $\ell$ by the same letters.

We also set $s_{m,n}\coloneqq \frac{n-m-\epsilon_0}{2}$ and for $s\in\frac{1}{2}\ZZ$ we define
\[I(s)\coloneqq \id_{P_{n}}^{G(W+W^-)}(\chi_V\lvert\det\rvert^{s}),\] where $P_n$ is seen as the stabilizer of the isotropic subspace $\Delta W\subseteq W+W^-$.
The following is due to Rallis in the case of $\CC$-coefficients and was extended to modular coefficients in \cite{Trias2025}.
\begin{lemma}
    We have $\Theta_{V,W+W^-,\psi,R}(\chi_W)\hra I(s_{m,n})$.
\end{lemma}
Next, we call $\ell$ strongly banal if it is banal with respect to $G(W+W^-)$, \emph{i.e.} $o(q)> 4n$.
Note that in the proof of \cite[Proposition 6.17]{Trias2025}, the author considers also the $\zl$-valued Weil representation $\omega_{\psi,\zl}$, and in particular equips $\Theta_{V,W+W^-,\psi,\ql}(\chi_W)$ with a natural integral structure $\Theta_{V,W+W^-,\psi,\zl}(\chi_W)$ such that
\[\Theta_{V,W+W^-,\psi,\zl}(\chi_W)\bl\cong \Theta_{V,W+W^-,\psi,\fl}(\chi_W).\]
We now come to our main contribution.
\begin{prop}\label{P:H}
    The representation $\Theta_{V,W+W^-,\psi,\fl}(\chi_W)$ is irreducible if $\ell$ is strongly banal.
\end{prop}
\begin{proof}
    We recall that $\Theta_{V,W+W^-,\psi,\ql}(\chi_W)$ is irreducible, see for example \cite[Proposition 7.2]{GanIch16}. Moreover, if $\ell$ is banal for $G(W+W^-)$, the cuspidal support of $I(s_{m,n})$ is banal, since the cuspidal support of $\lvert\det\rvert^{s_{m,n}}\chi_V$ is given by
    \[[\lvert-\rvert^{s_{m,n}+\frac{-n+1}{2}}\chi_V]+\ldots+[\lvert-\rvert^{s_{m,n}+\frac{n-1}{2}}\chi_V].\]
    Hence by \Cref{C:lifting}, the reduction of any irreducible subquotient of $I(s_{m,n})$ is irreducible. In particular, by the above observation, we have that 
    \[\Theta_{V,W+W^-,\psi,\fl}(\chi_W)\] is irreducible.
\end{proof}
The following statement was called Hypothesis (H) in \cite{Trias2025}.
\[I(-s_{m,n})\sra \Theta_{V,W+W^-,\psi,\fl}(\chi_W).\]
In the work of \cite{Trias2025} it was the main obstacle in proving a banal version of Howe-duality.
It is a straightforward consequence of \Cref{P:H} and the MVW-involution.
\begin{corollary}\label{C:H}
    Let $\ell$ be a strongly banal prime for $G(W)$ and $G(V)$. Then Hypothesis (H) holds.
\end{corollary}
As a consequence we obtain a proof of Howe duality for the type I dual pair $(G(W),G(V))$ if $\ell$ is a strongly banal prime. Note that under the assumption that Hypothesis (H) holds, the proof below can already be found in \cite{Trias2025}.
\begin{theorem}
    Let $\ell$ be a strongly banal prime for the pair $(G(W),G(V))$. Then for any $\pi\in\irr_{\fl}(G(W))$ the following holds.
    \begin{enumerate}
        \item $\Theta_{W,V,\psi,\fl}(\pi)$ is of finite length.
        \item The cosocle $\theta_{W,V,\psi,\fl}(\pi)$ of $\Theta_{W,V,\psi,\fl}(\pi)$ is irreducible or $0$.
        \item If $\pi'\in\irr_{\fl}(G(W))$ is such that \[\theta_{W,V,\psi,\fl}(\pi)\cong \theta_{W,V,\psi,\fl}(\pi')\neq 0,\] then $\pi\cong \pi'$.
    \end{enumerate}
\end{theorem}
\begin{proof}
    We recall the proof of Howe-duality of \cite{GanTakeda2016}. The proof uses the following ingredients.
    \begin{enumerate}
        \item Kudla's filtration of the Jacquet-module of $\omega_{\psi,\fl}$, see \cite{Kudla1986}.
        \item The See-Saw mechanism of \cite[§6.1]{GanIchino2014}.
        \item Hypothesis (H).
        \item The filtration of \cite{KudlaRallis2005} of the restriction of $I(-{s_{m,n}})$ to $G(W)\times G(W^-)$.
        \item The theory of derivatives.
    \end{enumerate}
We proved (3) in \Cref{C:H}, (5) was discussed in \Cref{S:der}, and (1), (2), and (4) were already treated in \cite{Trias2025}.
Thus one can prove the three claims of the theorem as in \cite{GanTakeda2016}.
\end{proof}
\bibliographystyle{abbrv}
\bibliography{References.bib}
\end{document}